\documentclass[12pt]{amsart}
\usepackage{amsfonts}
\usepackage{amsfonts,latexsym,rawfonts,amsmath,amssymb,amsthm}
\usepackage[plainpages=false]{hyperref}

\usepackage{graphicx}

\RequirePackage{color}

\numberwithin{equation}{section}

\newcommand{\beq}{\begin{equation}}
\newcommand{\eeq}{\end{equation}}
\newcommand{\beqs}{\begin{eqnarray*}}
\newcommand{\eeqs}{\end{eqnarray*}}
\newcommand{\beqn}{\begin{eqnarray}}
\newcommand{\eeqn}{\end{eqnarray}}
\newcommand{\beqa}{\begin{array}}
\newcommand{\eeqa}{\end{array}}

\def\p{\partial }

\def\R{\Bbb R}

\def\Om{\Omega}
\def\pom{\p  \Omega}

\newtheorem{Proposition}{Proposition}[section]
\newtheorem{Theorem}[Proposition]{Theorem}
\newtheorem{Lemma}[Proposition]{Lemma}

\title  { Global Regularity for  minimal graphs over convex domains in hyperbolic space}\thanks{This work was supported by NSFC   11771237.}
\begin{document}

\address{Huaiyu Jian: Department of Mathematics, Tsinghua University, Beijing 100084, China.}

\address{You Li:Department of Mathematics, Tsinghua University, Beijing 100084, China.}

%  \thanks {This work was supported by Australian Research Council
%  and NSFC10631020 and 10871061.}

%  \subjclass {35J60, 53A15, 53C45}

\email{hjian@tsinghua.edu.cn£¬ you-li16@mails.tsinghua.edu.cn. }

%\date{}

\bibliographystyle{plain}

%\tableofcontents
\maketitle

\baselineskip=15.8pt
\parskip=3pt

\centerline {\bf  Huaiyu Jian}
\centerline {Department of Mathematics, Tsinghua University}
\centerline {Beijing 100084, China}
 \centerline{( Email: hjian@tsinghua.edu.cn\ \ Tel: 86-10-62772864)}

\vskip10pt

\centerline { \bf You Li}
\centerline {Department of Mathematics, Tsinghua University}
\centerline {Beijing 100084, China}

\vskip20pt

\noindent {\bf Abstract}:
%\begin{abstract}
In this paper we study the global regularity
for the solution to the Dirichlet problem of the equation  of minimal graphs over a convex domain in hyperbolic spaces.
 We find that the global regularity   depends only on the convexity of the domain but
 independent of its smoothness.
 Basing on   the  invariance of the problem under translation and rotation transforms,
 we construct the super-solution to the problem, by which we prove the optimal and  accurate global regularity
 for this problem.

%\end{abstract}

 \vskip20pt


\vskip20pt

\noindent {\bf  Running head}: Regularity for  Minimal Graphs

\vskip20pt

\baselineskip=15.8pt
\parskip=3pt

\newpage

\centerline {\bf     Global Regularity for   Minimal Graphs over convex domains in Hyperbolic Space }

 \vskip10pt

\centerline { Huaiyu Jian\ \ \ \ You Li}

%\tableofcontents
\maketitle

\baselineskip=15.8pt
\parskip=3.0pt

\section { Introduction}
%sec 1

  In this paper we consider  the Dirichlet problem
\beq \label {1.1}
\begin{split}
\Delta u-\frac{u_{i}u_{j}}{1+|\nabla u|^{2}}u_{ij}+\frac{n}{u}&= 0\ \  \text{in}\ \Om,\\
   u&=0\ \ \text{on}\ \pom,\\
   u&>0\ \  \text{in}\ \Om,
\end{split}
\eeq
where $\Om$ is a bounded   domain in $\R^n$ ($n\geq 2$).  In \cite{[L1]}, Lin
proved that problem (1.1) admits a unique solution $u\in C(\overline{\Om})\bigcap C^{\infty}(\Om)$ if
$\partial \Om\in C^2$ and $H_{\partial \Om}\geq 0$, where $H_{\partial \Om} $ denotes the mean curvature of
$ \partial \Om $ with respect to the inward normal direction. Furthermore, he proved that $u\in C^{\frac{1}{2}}(\overline{\Om}) $ if
$\partial \Om\in C^2$ and $H_{\partial \Om}>0$.  Anderson \cite{[A1], [A2]},
Hardt and Lin \cite{[HL]}, and Lin \cite{[L2]} studied the various problems related the complete area-minimizing submanifolds in hyperbolic spaces.
These results were partially extended by Tonegawa \cite{[T]} and Guan-Spruck \cite{[GS]} to the constant mean curvature case.

Recently, Han, Shen and Wang in \cite{[HSW]} studied the  optimal regularity for problem (1.1) and proved the following theorem. See Theorems 1.2 and  3.1
in \cite{[HSW]}.

\begin{Theorem} \label {1.1}
  Assume that $n\geq 2$, and $\Om\subseteq R^n$ is a bounded convex domain.  Then we have

  {\bf (i)}\,  (1.1) admits
a unique solution $u\in C^{\frac{1}{n+1}}(\overline{\Om})\bigcap C^{\infty}(\Om)$,   $u$ is concave,  and
 $$[u]_{C^{\frac{1}{n+1}}(\overline{\Om})}\leq[(n+1)d_{\Om}^{n}]^{\frac{1}{n+1}},$$
 where $d_{\Om}$  denotes the diameter of $\Om$;

  {\bf (ii)}\, $u\in C^{\frac{1}{2}}(\overline{\Om})$ if $\Om$ is the intersections of finitely many bounded convex $C^2$-domains
  $\Om_i$ with $H_{\partial \Om_i} >0$.
\end{Theorem}

In this paper we continue to study the global H$\ddot{o}$lder regularity  for  the concave solutions to  problem (1.1), using a method
different  from \cite{[HSW]}.
We find that the H$\ddot{o}$lder exponent of the concave solution is independent of the smoothness of $\Om$ but depends on
the convexity of the domain. This fact was found first by us in \cite{[JL]} for a class of singular Monge-Ampere equations.
   To describe the convexity,
we  use the concept of  $(a, \eta)$- type for bounded convex domain introduced in \cite{[JL]}.

 From now on,  denote
 $$x'=(x_{1},...,x_{n-1}), \ \  |x'|=\sqrt{x_{1}^{2}+...+x_{n-1}^{2}}.$$

 \noindent{\bf  Definition 1.1}\cite{[JL]}. {\sl
Let $\Om\subseteq R^n$ be a bounded convex domain,   $x_{0}\in \pom$, $a\in[1,+\infty)$ and  $\eta>0$.    We say $x_{0}$ is $(a, \eta)$- type point if by   translation and rotation transforms, we have
  \beq \label{1.2} x_0=0, \ \  \Om\subseteq \{(x, x_n)\in R^n | x_n \geq\eta|x'|^{a} \}.\eeq
    $\Om$  is called as $(a, \eta)$- type domain if  for all $x\in\pom$,  $x$ is $(a, \eta)$- type point;
     $\Om$  is called as $(+\infty, \eta)$- type domain if   it is not  $(a, \eta)$- type for any $a\geq 1$ and any $\eta>0$. }

Obviously,  a domain which boundary   contains
pieces of (flat) hyperplane is $(+\infty, \eta)$- type.

\noindent{\bf Remark 1.1}. The convexity requires that the number $a$ should be no less than 1. Since $\Om $  is   bounded and convex, we see that
it is sufficient (1.2) holds true only for all $|x'|<\mu$ and some small $\mu>0$, and that if $a_1>a_2$,  $(a_2, \eta_2)$- type domains
are also $(a_1, \eta_1)$- type for some $\eta_1=C(d_{\Om}, n, a_1, a_2, \eta_2)>0$, as the well-known fact that a $C^{k_1}$-domain is also $C^{k_2}$-domain
if $k_1>k_2$. Finally,  there is no $(a, \eta)$ type domain for $a\in (1, 2)$, although the boundary of a convex domain may contains  $(a, \eta)$-  type points.

\vskip 0.5cm
The main result of this paper is stated as the following theorem.

\begin{Theorem} \label {1.2}  Suppose that $\Om\subseteq R^n$ is a bounded convex domain and  $u\in C(\overline{\Om})\bigcap C^{\infty}(\Om)$  is the concave solution to  problem (1.1).
  If $\Omega$ is $(a, \eta)$- type domain with $a\in [2,+\infty]$,
then $u\in C^{\bar a}(\overline{\Omega})$
and
\beq \label{1.3} |u|_{C^{\bar a}(\overline{\Omega})}\leq C ,\eeq
 where $\bar a =max\{ \frac{1}{a}, \frac{1}{n+1}\}$ and $C=C(a, \eta,   d_{\Omega}, n)$ is a positive constant depending only on $a, \eta,  n$ and $d_{\Omega}$.
 In particular,
\beq \label{1.3} C=2\sqrt{2Rd_{\Om}} \, \text{when} \, a=2, \ \   C=2(n+1)^2d_{\Om}^{\frac{1}{n+1}} \, \text{when} \, a=+\infty,\eeq
where $R$ is the exterior sphere radius of $\Om$ (see Definition 2.2 below).

\end{Theorem}

 We would like to point out that the H$\ddot{o}$lder exponent $\bar a$ can not be large than $\frac{1}{n+1}$ if  $\Omega$ is $(+\infty, \eta)$- type (see Remark
 2.3 in \cite{[HSW]} for a domain which boundary   contains
pieces of (flat) hyperplane ), and it   can not be large than $\frac{1}{2}$ if $\Om $ is $(2, \eta)$- typea (see Lemma 2.5 below for a ball). In this view, our
Theorem 1.2  should be optimal and accurate. Unfortunately, the case $a=2$ and $n\geq 3$ of Theorem 1.2 does not include the (ii) of Theorem 1.1, since the domain
$\Om$ described there is not necessarily $(2, \eta)$-type, although the it is  $(2, \eta)$-type when $n=2$. Interestingly, the problem (1.1) for $n=2$ also appears
in the study of two-dimensional Riemann problems of the Chaplying gas. See section 6 in \cite{[S]}.

  For certain points near a $(a, \eta)$ type boundary point for $a\in [1, 2)$, we have the following local estimates.

 \begin{Theorem} \label{1.3}
Suppose that $\Om\subseteq R^n $  is a bounded  convex domain,  $u$ is the concave solution to problem (1.1), and
  $x_{0}\in\pom$ is (a, $\eta$) type point with $a\in [1,2)$.  Without loss of generality, we assume
  \beq \label {1.4} x_{0}=0, \ \  \Om\subseteq \{(x,x_n)\in R^{n}|\,
  x_n \geq\eta|x'|^{a} \}.
 \eeq
  Then for any  $\delta >0$, there exists a positive constant $C=C(a, \eta,   d_{\Om}, n, \delta)$
such that
\beq \label{1.5} |u(x)|\leq C |x|^{\frac{1}{a+\delta}},\ \  \forall  x\in \Om\bigcap\{(x',x_n)\in R^n: x'=0\}.\eeq
\end{Theorem}
\vskip 0.5cm

This paper is arranged as follows. In Section 2 we   construct a class of non-smooth strictly convex domain which is $(2, \eta)$ type, and  prove Theorem 1.2 for the case of $a=2$.
 In Section 3 we   construct super-solutions   and    prove Theorem 1.2  for the case of $a\in (2,\infty]$.    In Section 4,  we  prove Theorem 1.3.
   We should point out that applying the invariance of problem (1.1)   under translation and rotation transforms to construct super-solutions  and using the convexity sufficiently is critical for our arguments, which are
    different from the  proof of Theorem 1.1 in  \cite{[HSW]}.

\vskip20pt
%\newpage

\baselineskip=15.6pt
\parskip=2.5pt

\section {The case $a=2$ }
%sec 2
For any $\widetilde{x}\in\pom$, a hyperplane  $ {\bf P}$  is called a generalized tangent hyperplane (supported hyperplane) at   $\widetilde{x}$ if all points of $\Om$ is on the same side
 of  $ {\bf P} $ and $\overline{\Om}\bigcap  {\bf P}\ni \{\widetilde{x}\}$. Obviously, tangent hyperplane at a boundary point is not necessarily unique, unless $\Om \in C^1$.  But we can choose a tangent hyperplane  $  {\bf P}$ such that  part of $\pom$ near the $\widetilde{x}$  is expressed by a function with
 respect to tangent variable $x'\in   {\bf P}$. We call such a function as  {\sl a tangent expression of $\pom$ near the $\widetilde{x}$}, denoted by $T_{\widetilde{x}}\Om(x')$.

 If $\Om\in C^2$ and $ \widetilde{x}\in\pom$,  the tangent plane at $ \widetilde{x} $ is unique which is denoted  by $T_{\widetilde{x}}\pom$, and the tangent expression  $T_{\widetilde{x}}\Om(x')$ is locally in  $C^2$.   We see that for  a bounded convex $C^2$- domain $\Om$,   there is a constant $ \lambda$  such that  for any $\widetilde{x}\in\pom$ and any tangent expression  $T_{\widetilde{x}}\Om(x')$
$$D_{\gamma\gamma}T_{\widetilde{x}}(\widetilde{x})\geq \lambda ,    \ \  \forall  \gamma\in B(0)\subset T_{\widetilde{x}}\pom .$$
Let {\sl
$\lambda(\pom)$ be the maximal number among all such $\lambda$.}

 We  are going  to construct  a class of non-smooth strictly convex domains which is $(2, \eta)$-type.
\vskip 0.5cm

\noindent{\bf  Definition 2.1}. {\sl If $\Om$ is  a bounded domain  and there is a $ \lambda>0$ and a  sequence of $C^2$ domains,   
$\{\Om_{i}\}_{i=1}^{\infty}$,  satisfying
$\Om_{i}\subseteq\Om_{i+1}, \  \lambda(\pom_{i})\geq\lambda \ \ \forall \,i $,  such that
$$ \Om=\bigcup_{i=1}^{\infty}\Om_{i} ,$$
then  $\Om$ is called as a   $\lambda$-convex domain.
  The maximal number among all such $\lambda$ is denoted by $\lambda(\pom)$.}

 Obviously, a  $\lambda$-convex domain $\Om$ may contain singular (angular) points.

\begin{Lemma} \label{2.1}
Assume $\Om$ is a  $\lambda$-convex domain and $x\in\pom$.
Without loss of generality (by translation or rotation transforms), we assume $x=0$ and $\Om\subseteq R_{+}^{n}:=\{(x,x_n)\in R^n|x_n>0\}$.
 Then $\Om\subseteq B_{R}(Re_{n})$, the ball centered in $Re_{n}$ with radius $R$,
where $R=\frac{1}{\lambda(\pom)}$ and $e_{n}=(0, 0, \cdots, 0, 1)$.
\end{Lemma}
\begin{proof}  We will complete the proof by two steps.

{\sl Step 1.}   Assuming that $\Om$ is smooth, we  are going to prove that for any $ \delta>0$
 \beq \label{2.1}  \Om\subseteq B_{R+\delta}((R+\delta)e_{n}) .\eeq
 It is enough to   prove that  for any two-dimensional plane $P$ which contains $x_{n}-$axis, we have
  \beq \label{2.2} \Om\bigcap P\subseteq B_{R+\delta}((R+\delta)e_{n})\bigcap P.\eeq

 Note that $\Om\bigcap P$ and $B_{R+\delta}((R+\delta)e_{n})\bigcap P$ is two dimensional domain in the two-dimensional plane $P$,
 touched at $x=0$.

  Write $$\gamma_{1}=\partial\Om\bigcap P, \ \ \gamma_{2}=\partial B_{R+\delta}((R+\delta) e_{n})\bigcap P.$$
 We see that $\gamma_{1}$ and $\gamma_{2}$ is two curves in plane $P$ and we have
 $curvature(\gamma_{1})\geq\lambda(\pom)=\frac{1}{R}>\frac{1}{R+\delta}=curvature(\gamma_{2})$. Hence,
 $\gamma_{1}$  is above (in the inside of) $\gamma_{2}$ near $x=0$ and they are tangent at $x=0$.

 To prove (2.2), it is sufficient to prove the claim that $\gamma_{1}$ is always above (in the inner side of) $\gamma_{2}$.

 Suppose that the claim is false. We let $\overline{x}\neq0$  be the  nearest contact point  from $0$ and introduce a new coordinate in plane $P$ as follows.
  Take line 0$\overline{x}$ to be the coordinate axis with variable t
and take the direction orthogonal
to the line $0\bar{x}$ to be the graph height  coordinate for $\gamma_{1}$ and $\gamma_{2}$.
 Without loss of generality, we assume $0\geq\gamma_{1}(t)\geq\gamma_{2}(t)$ for $  t\in [0,   \overline{x}]$.

 Now we consider the point $t_{0}$ such that
  $$\gamma_{1}(t_0)-\gamma_{2}(t_0)=\max_{t\in [0, \bar x]}(\gamma_{1} -\gamma_{2})(t).$$
  Then $t_{0}\in (0, \bar x) $, $(\gamma_{1} -\gamma_{2} )'(t_{0})=0$ and $(\gamma_{1}-\gamma_{2})''(t_{0})\leq0$.
 That is
  $$\gamma_{1}'(t_{0})=\gamma_{2}'(t_{0}), \ \   0\leq\gamma_{1}''(t_{0})\leq\gamma_{2}''(t_{0}),$$
 which implies that $curvature(\gamma_{1})(t_{0})\leq curvature(\gamma_{2})(t_{0})$, a contradiction!

In this way, we have proved (2.1).  Letting $\delta$ to 0, we have
 \beq \label{2.3} \Om\subseteq B_{R}(Re_{n}).\eeq

{\sl Step 2.}   For  a general   $\lambda$-convex domain $\Om$, by Definition 2.1 and the result of Step 1, one easily see that (2.3) still holds.

\end{proof}

\noindent{\bf  Definition 2.2}. {\sl We say that a   domain
  $\Om$   in $R^n$ satisfies exterior sphere condition with radius $R$ if   for  each $x\in \pom$,  there
 is a ball $B_R(y)$ centered at $y$ with radius $R$, such that $B_R(y) \supseteq \Om $ and  $\partial B_R(y) \bigcap  \pom\ni x$.}

It is direct from Lemma 2.1 that
\begin{Lemma} \label{2.2} A  $\lambda$-convex domain   satisfies  exterior sphere condition with radius $R=\frac{1}{\lambda(\pom)}$.
 \end{Lemma}

 \begin{Lemma} \label{2.3}
  A bounded  convex domain  satisfying exterior sphere condition with radius $R$
is  $(2, \frac{1}{2R})$- type domain.
Conversely, a $(2, \eta)$- type domain satisfies exterior sphere condition with radius $R=max\{\frac{1}{\eta},\ d_{\Omega}\}$.
 \end{Lemma}

 \begin{proof}  See Lemma 2.1 in \cite{[JL]}.
\end{proof}

Therefore, $\lambda$-convex domains are $(2, \eta)$- type.

\begin{Lemma}\label{2.4}
Let $\Om$ be a bounded convex domain and $u$ be a  concave  function with $u|_{\pom}=0$. If there are $\alpha\in(0,1]$ and  $M>0$
such that
\beq \label {2.4} |u(x)|\leq M{d_{x}}^{\alpha},\ \ \forall x\in \Om\eeq
where  $d_{x}=dist(x,\pom)$, then $u\in C^{\alpha}(\overline{\Om})$ and
$$|u|_{C^{\alpha}(\overline{\Om})}\leq 2 M d_{\Om}^{\alpha}.$$
 \end{Lemma}
\begin{proof} See Lemma 2.3 in \cite{[JL]}.
\end{proof}

\begin{Lemma} %2.5
When $\Om=B_{R}(0)\subset R^n$,  then $U(x)=\sqrt{R^{2}-|x|^{2}}$ is the solution to problem (1.1).
\end{Lemma}

\begin{proof}
In the spherical symmetry case, $u(x)=u(r)$ where $r=|x|$. Then the  equation in (1.1)  is reduced to
\beq \label{2.5}
\begin{split}
(n-1)\frac{u_{r}}{r}+\frac{u_{rr}}{1+u_{r}^{2}}+\frac{n}{u}=0.
\end{split}
\eeq

By a simple computation, we obtain
\beq  \label{2.6}
\begin{split}
U_{r}=&\frac{-r}{\sqrt{R^{2}-r^{2}}},\\
\frac{U_{r}}{r}=&\frac{-1}{\sqrt{R^{2}-r^{2}}}, \\
U_{rr}=&\frac{-R^{2}}{(R^{2}-r^{2})\sqrt{R^{2}-r^{2}}}.
\end{split}
\eeq
Using (2.6), we have
 $$(n-1)\frac{U_{r}}{r}+\frac{U_{rr}}{1+U_{r}^{2}}+\frac{n}{U}=  0. $$

\end{proof}

\begin{Theorem}\label{2.6}
Suppose $\Om$ is a bounded  convex domain satisfying exterior sphere condition with radius $R$ and $u$ is the
concave solution to  problem (1.1).
Then $u\in C^{\frac{1}{2}}(\overline{\Om})$,
and $|u|_{C^{\frac{1}{2}}(\overline{\Om})}\leq 2\sqrt{2Rd_{\Om}}$.
\end{Theorem}

\begin{proof}
Due to Lemma 2.4,  it is enough to
 prove $$ u(y)\leq \sqrt{2R} {d_{y}}^{\frac{1}{2}}, \ \ \forall y\in \Om .$$
 Taking a $y\in \Om$, we can find $z\in \pom$ such that $dist(y,z)=d_{y}.$
 Without loss generality (by translation and rotation), we assume $z=0$ and
  the line determined by $z$ and $y$
 is the $x_{n}-axis$. Note that the equation in problem (1.1) is invariant under translation and rotation transforms.

 Since $\Om$ is a bounded  convex domain satisfying exterior sphere condition with radius $R$,
 we conclude that $\Om\subseteq B_{R}(Re_{n})$. Let $U$ be the solution of (1.1) in $B_{R}(Re_{n})$ as in Lemma 2.5.
 By comparison principal, we conclude $u\leq U$ in $\Om$. Restricting on the point $y$ we have
\begin{equation*}
\begin{split}
u(y)\leq  U(y)=&\sqrt{R^{2}-(R-d_{y})^2}\\
=&\sqrt{2Rd_{y}-d_{y}^2}\\
\leq& \sqrt{2R}\sqrt{d_{y}} .
\end{split}
 \end{equation*}

\end{proof}

{\bf Proof of the case $a=2$ of Theorem 1.2:} \  it is direct from Lemma 2.3 and Theorem 2.6.

\section{ Proof of Theorem 1.2}
%sec 3

In previous Section we have proved    Theorem 1.2 for the case $a=2$. In this section,  we prove   Theorem 1.2 for the case $a\in(2, \infty]$
and thus complete its proof.

\begin{Theorem}\label {3.1}
Let $\Om$ be $(a, \eta)$ type domain with $a\in (2,+\infty)$  and $u$ is the concave solution to problem (1.1).
Then $u\in C^{\frac{1}{a}}(\overline{\Om})$,
and $|u|_{C^{\frac{1}{a}}(\overline{\Om})}\leq C(a, \eta,   d_{\Om}, n)$.
\end{Theorem}

\begin{proof}

By Lemma 2.4, it is sufficient  to prove
\beq \label{3.1}
|u(y)|\leq C(a,n, \eta,  d_{\Om})\ {d_{y}}^{\frac{1}{a}}, \ \ \forall y\in \Om .
\eeq

 For any $y\in \Om$, we can find $z\in \pom$, such that $dist(y,z)=d_{y}.$
 Since the domain $\Om$ is $(a,\eta)$- type,
 without loss generality (by translation and rotation), we may assume $z=0$, and take the line determined by $z$ and $y$
  as the $x_{n}-axis$ such that
 \beq \label{3.2}
 \Om\subseteq \{(x, x_n)\in R^{n}|   x_n\geq\eta|x'|^{a}  \} .
 \eeq
We should point out that  problem (1.1) is invariant under translation and rotation transforms.

 Let
  \beq \label{3.3}W(x_{1}, ..., x_{n})=((\frac{x_{n}}{\varepsilon})^{\frac{2}{a}}-x_{1}^{2}-...-x_{n-1}^{2})^{\frac{1}{b}},
  \eeq
   where  $b\geq 2$ is a constant to be fixed.  We will choose  a  $\varepsilon>0$ such that $W$ is a super-solution to problem (1.1).

Observing that (3.2) implies that $\pom$ lies over the hypersurface $x_n=\eta|x'|^{a}$,  we can find a small $\varepsilon=C(a, \eta, d_{\Om}, n)>0$ such that
  \beq \label{3.4}
  W\geq u=0\ \  \text {on}   \ \  \partial\Om .
  \eeq
 For brevity, we write
 $$r=|x'|=\sqrt{x_{1}^{2}+...+x_{n-1}^{2}}, \ \ W(x)=W(r,x_{n})$$ and
 $$W_{i}=W_{x_{i}} , \ \  W_{ij}=W_{x_{i}x_{j}}$$
  for $i, j\in\{1, 2, ..., n-1\}$. A direct computation yields
\beq \label{3.5}
\begin{split}
W_{i}=&W_{r}\frac{x_{i}}{r},\\
W_{ij}=&W_{rr}\frac{x_{i}}{r}\frac{x_{j}}{r}+W_{r}\frac{\delta_{ij}r-\frac{x_{i}x_{j}}{r}}{r^{2}},\\
 =&\frac{W_{r}}{r}\delta_{ij}+(W_{rr}-\frac{W_{r}}{r})\frac{x_{i}}{r}\frac{x_{j}}{r},\\
W_{in}=&W_{rn}\frac{x_{i}}{r}.\\
\end{split}
\eeq
Let $F[u]$ denote the left hand side of the equation in  (1.1), i.e.,
\beq \label {3.6}
F[u]=\Delta u-\frac{u_{i}u_{j}}{1+|\nabla u|^{2}}u_{ij}+\frac{n}{u}.
\eeq
Then by (3.5) we have
\beq \label{3.7}
\begin{split}
F[u]&=(n-2)\frac{W_{r}}{r}+W_{rr}+W_{nn}+\frac{n}{W}\\
& \ \ -\frac{W_{rr}\cdot W_{r}^{2}+2W_{rn}\cdot W_{r}\cdot W_{n}+W_{nn}\cdot W_{n}^{2}}{1+W_{r}^{2}+W_{n}^{2}}\\
 &=[  W_{rr}\cdot(1+W_{n}^{2})+W_{nn}\cdot(1+W_{r}^{2})-2W_{rn}\cdot W_{r}\cdot W_{n}\\
 &+((n-2)\frac{W_{r}}{r}+\frac{n}{W})
  (1+W_{r}^{2}+W_{n}^{2})]\cdot
 [1+W_{r}^{2}+W_{n}^{2}]^{-1}\\
 &:=\frac{I+J}{1+W_{r}^{2}+W_{n}^{2}}
\end{split}
\eeq
where
\begin{equation*}
\begin{split} I&=W_{rr}\cdot(1+W_{n}^{2})+W_{nn}\cdot(1+W_{r}^{2})-2W_{rn}\cdot W_{r}\cdot W_{n},\\
 J&=((n-2)\frac{W_{r}}{r}+\frac{n}{W})\cdot(1+W_{r}^{2}+W_{n}^{2}).
\end{split}
 \end{equation*}
By the expression (3.3), we compute
\begin{equation*}
\begin{split}
W_{r}&=-\frac{2}{b}W^{1-b}\cdot r,\\
W_{n}&=\frac{2}{ab}W^{1-b}\cdot(\frac{x_{n}}{\varepsilon})^{\frac{2}{a}-1}\cdot \frac{1}{\varepsilon},\\
W_{rr}&=\frac{4(1-b)}{b^{2}}W^{1-2b}\cdot r^{2}+(-\frac{2}{b})W^{1-b},\\
W_{nn}&=\frac{4(1-b)}{a^{2}b^{2}}W^{1-2b}\cdot (\frac{x_{n}}{\varepsilon})^{\frac{4}{a}-2}\cdot(\frac{1}{\varepsilon})^{2}+\frac{2(2-a)}{a^{2}b}W^{1-b}\cdot (\frac{x_{n}}{\varepsilon})^{\frac{2}{a}-2}\cdot(\frac{1}{\varepsilon})^{2},\\
W_{rn}&=\frac{4(b-1)}{ab^{2}}W^{1-2b}\cdot (\frac{x_{n}}{\varepsilon})^{\frac{2}{a}-1}\cdot r \cdot \frac{1}{\varepsilon}.
\end{split}
 \end{equation*}
Hence
\begin{equation*}
\begin{split}
W_{rr}\cdot(1+W_{n}^{2}) = & \frac{4(1-b)}{b^{2}}W^{1-2b}\cdot r^{2} -\frac{2}{b}W^{1-b}
+\frac{16(1-b)}{a^{2}b^{4}}\cdot W^{3-4b}\cdot r^{2}\cdot(\frac{x_{n}}{\varepsilon})^{\frac{4}{a}-2}\cdot(\frac{1}{\varepsilon})^{2} \\
 & -\frac{8}{a^{2}b^{3}}\cdot W^{3-3b}\cdot(\frac{x_{n}}{\varepsilon})^{\frac{4}{a}-2}\cdot(\frac{1}{\varepsilon})^{2},
\end{split}
 \end{equation*}
 \begin{equation*}
 \begin{split}
W_{nn}\cdot(1+W_{r}^{2})=& \frac{4(1-b)}{a^{2}b^{2}}W^{1-2b}\cdot (\frac{x_{n}}{\varepsilon})^{\frac{4}{a}-2}\cdot(\frac{1}{\varepsilon})^{2}
+\frac{2(2-a)}{a^{2}b}W^{1-b}\cdot (\frac{x_{n}}{\varepsilon})^{\frac{2}{a}-2}\cdot(\frac{1}{\varepsilon})^{2}\\
& +\frac{16(1-b)}{a^{2}b^{4}}\cdot W^{3-4b}\cdot r^{2}\cdot(\frac{x_{n}}{\varepsilon})^{\frac{4}{a}-2}\cdot(\frac{1}{\varepsilon})^{2}\\
& +\frac{8(2-a)}{a^{2}b^{3}}\cdot W^{3-3b}\cdot r^{2}\cdot(\frac{x_{n}}{\varepsilon})^{\frac{2}{a}-2}\cdot(\frac{1}{\varepsilon})^{2},
\end{split}
 \end{equation*}
and
$$
 -2W_{rn}\cdot W_{r}\cdot W_{n}
= \frac{-32(1-b)}{a^{2}b^{4}}\cdot W^{3-4b}\cdot r^{2}\cdot(\frac{x_{n}}{\varepsilon})^{\frac{4}{a}-2}\cdot(\frac{1}{\varepsilon})^{2}.$$

Adding the above three equalities, we obtain
 \begin{equation*}
 \begin{split}
 I=&W_{rr}\cdot(1+W_{n}^{2})+W_{nn}\cdot(1+W_{r}^{2})-2W_{rn}\cdot W_{r}\cdot W_{n}\\
=&\frac{4(1-b)}{b^{2}}W^{1-2b}\cdot r^{2}-\frac{2}{b}W^{1-b}-\frac{8}{a^{2}b^{3}}\cdot W^{3-3b}\cdot(\frac{x_{n}}{\varepsilon})^{\frac{4}{a}-2}\cdot(\frac{1}{\varepsilon})^{2}\\
 & +\frac{4(1-b)}{a^{2}b^{2}}W^{1-2b}\cdot (\frac{x_{n}}{\varepsilon})^{\frac{4}{a}-2}\cdot(\frac{1}{\varepsilon})^{2}
+\frac{2(2-a)}{a^{2}b}W^{1-b}\cdot (\frac{x_{n}}{\varepsilon})^{\frac{2}{a}-2}\cdot(\frac{1}{\varepsilon})^{2}\\
 & +\frac{8(2-a)}{a^{2}b^{3}}\cdot W^{3-3b}\cdot r^{2}\cdot(\frac{x_{n}}{\varepsilon})^{\frac{2}{a}-2}\cdot(\frac{1}{\varepsilon})^{2}.\\
\end{split}
 \end{equation*}

 Since
$$1+W_{r}^{2}+W_{n}^{2}=1+\frac{4}{b^{2}}\cdot W^{2-2b}\cdot r^{2}+\frac{4}{a^{2}b^{2}}\cdot W^{2-2b}\cdot(\frac{x_{n}}{\varepsilon})^{\frac{4}{a}-2}(\frac{1}{\varepsilon})^{2}$$
and
$$(n-2)\cdot\frac{W_{r}}{r}+\frac{n}{W}=(n-2)(-\frac{2}{b})W^{1-b}+nW^{-1},$$
 we have
 \begin{equation*}
 \begin{split}
 J&=((n-2)\frac{W_{r}}{r}+\frac{n}{W})\cdot(1+W_{r}^{2}+W_{n}^{2})\\
&=(2-n)(\frac{2}{b})W^{1-b}+(n-2)\frac{-8}{b^{3}}W^{3-3b}\cdot r^{2}+(2-n)\frac{8}{a^{2}b^{3}}W^{3-3b}\cdot(\frac{x_{n}}{\varepsilon})^{\frac{4}{a}-2}\cdot(\frac{1}{\varepsilon})^{2}\\
 & +nW^{-1}+\frac{4n}{b^{2}}|W|^{1-2b}\cdot r^{2}+\frac{4n}{a^{2}b^{2}}|W|^{1-2b}\cdot(\frac{x_{n}}{\varepsilon})^{\frac{4}{a}-2}\cdot(\frac{1}{\varepsilon})^{2}.
\end{split}
 \end{equation*}

   Therefore, we obtain
\begin{equation} \label{3.8}
 \begin{split}  I+J=&\frac{4(n+1-b))}{b^{2}}W^{1-2b}\cdot r^{2}+(1-n)(\frac{2}{b})W^{1-b}\\
 & +(1-n)(\frac{8}{a^{2}b^{3}})W^{3-3b}\cdot(\frac{x_{n}}{\varepsilon})^{\frac{4}{a}-2}\cdot(\frac{1}{\varepsilon})^{2}\\
 &+\frac{4(n+1-b))}{a^{2}b^{2}}W^{1-2b}\cdot(\frac{x_{n}}{\varepsilon})^{\frac{4}{a}-2}\cdot(\frac{1}{\varepsilon})^{2}
+\frac{2(2-a)}{a^{2}b}W^{1-b}\cdot(\frac{x_{n}}{\varepsilon})^{\frac{2}{a}-2}\cdot(\frac{1}{\varepsilon})^{2}\\
& +\frac{8(2-a)}{a^{2}b^{3}}W^{3-3b}\cdot r^{2}\cdot(\frac{x_{n}}{\varepsilon})^{\frac{2}{a}-2}\cdot(\frac{1}{\varepsilon})^{2}
+(2-n)(\frac{8}{b^{3}})W^{3-3b}\cdot r^{2}+nW^{-1}\\
:=& J_{1}+J_{2}+J_{3}+J_{4}+J_{5}+J_{6}+J_{7}+J_{8}.
\end{split}
 \end{equation}

 Observe the pairs of $J_1$ and $J_7$,  $J_2$ and $J_8$ , $J_3$ and $J_4$.
 Each pair can be combined or canceled if we take $b=2$.   Now we fix $b=2$. Consequently, we have
 \begin{equation*}\label {3.10}
 \begin{split}
J_{1}&=(n-1)W^{-3}\cdot r^{2},\\
J_{2}&=-(n-1)W^{-1},\\
J_{3}&=\frac{-(n-1)}{a^{2}}W^{-3}\cdot(\frac{x_{n}}{\varepsilon})^{\frac{4}{a}-2}\cdot(\frac{1}{\varepsilon})^{2},\\
J_{4}&=\frac{(n-1))}{a^{2}}W^{-3}\cdot(\frac{x_{n}}{\varepsilon})^{\frac{4}{a}-2}\cdot(\frac{1}{\varepsilon})^{2},\\
J_{5}&=\frac{(2-a)}{a^{2}}W^{-1}\cdot(\frac{x_{n}}{\varepsilon})^{\frac{2}{a}-2}\cdot(\frac{1}{\varepsilon})^{2},\\
J_{6}&=\frac{(2-a)}{a^{2}}W^{-3}\cdot r^{2}\cdot(\frac{x_{n}}{\varepsilon})^{\frac{2}{a}-2}\cdot(\frac{1}{\varepsilon})^{2},\\
J_{7}&=-(n-2)W^{-3}\cdot r^{2},\\
J_{8}&=nW^{-1}.
\end{split}
 \end{equation*}

  Hence,
 \begin{equation*}
 \begin{split}
 J_{1}+J_{7}&=W^{-3}\cdot r^{2},\\
J_{2}+J_{8}&=W^{-1},\\
J_{3}+J_{4}&=0.
\end{split}
 \end{equation*}

Recalling that
 $$W (x_{1}, ..., x_{n})=((\frac{x_{n}}{\varepsilon})^{\frac{2}{a}}-x_{1}^{2}-...-x_{n-1}^{2})^{\frac{1}{2}},$$
we have
 $$W^{2}=(\frac{x_{n}}{\varepsilon})^{\frac{2}{a}}-r^{2},$$
which implies that
  $$J_{1}+J_{7}+J_{2}+J_{8}=W^{-3}(\frac{x_{n}}{\varepsilon})^{\frac{2}{a}}$$
and
$$J_{5}+J_{6}=\frac{2-a}{a^{2}}\cdot W^{-3}(\frac{x_{n}}{\varepsilon})^{\frac{2}{a}}\cdot (\frac{x_{n}}{\varepsilon})^{\frac{2}{a}-2}(\frac{1}{\varepsilon})^{2}.$$

Therefore, we obtain
$$I+J= W^{-3}(\frac{x_{n}}{\varepsilon})^{\frac{2}{a}}[1+\frac{2-a}{a^{2}}\cdot x_{n}^{\frac{2}{a}-2}(\frac{1}{\varepsilon})^{\frac{2}{a}}] .$$

Since $a\in (2, +\infty)$,   $2-a<0$ and $\frac{2}{a}-2<-1<0$, we have
\begin{equation*}
 \begin{split}
 (1+W_{r}^{2}+W_{n}^{2})F[u]&= I+J\\
 & \leq W^{-3}(\frac{x_{n}}{\varepsilon})^{\frac{2}{a}}[1+\frac{2-a}{a^{2}}\cdot   d_{\Om}^{\frac{2}{a}-2}(\frac{1}{\varepsilon})^{\frac{2}{a}}].
\end{split}
 \end{equation*}

 Finally, choosing a smaller $\varepsilon =C(a, d_{\Om} )>0$ (if necessary) such that
$$1+\frac{2-a}{a^{2}}\cdot d_{\Om}^{\frac{2}{a}-2}(\frac{1}{\varepsilon})^{\frac{2}{a}}\leq 0,$$
  we have
  $$ F[W]\leq 0 \ \ \text{in} \ \ \Om .$$
  By this and (3.4),  we have proved that $W$ is an super-solution to problem (1.1).
      By comparison principal  we have
$$0\leq u\leq W   \ \ \text{in} \ \ \Om.$$
In particular,  for all $(\mathbf{0},x_{n})\in \Om $ we have
  \beq \label{3.9}
 0\leq u(\mathbf{0},x_{n})\leq W(\mathbf{0},x_{n}) =(\frac{1}{\varepsilon(a, \eta, d_{\Om}, n))})^{\frac{1}{a}}x_{n}^{\frac{1}{a}}.
 \eeq

Note that $ d_{y}=y_n$ and $y=(0, y_n)$ by the choice of the coordinate in the beginning. It follows from (3.9) that
 $$0\leq u(y) \leq(\frac{1}{\varepsilon(a, \eta,  d_{\Om}, n))})^{\frac{1}{a}}d_{y}^{\frac{1}{a}}, $$
 which proves (3.1) and thus completes the proof of Theorem 3.1.
\end{proof}

\begin{Theorem} \label {3.2}
   If $\Om$ is a general bounded convex domain  and $u$ is the concave solution to problem (1.1), then $u\in C^{\frac{1}{n+1}}(\overline{\Om})$ and
 \beq \label{3.10} [u]_{C^{\frac{1}{n+1}}(\overline{\Om})}\leq 2(n+1)^2d_{\Om}^{\frac{1}{n+1}}.\eeq
  \end{Theorem}

 \begin{proof}
 This is the   result (i) of   Theorem 1.1 except for a different H\"older norm, but our argument provides another proof.

 To prove (3.10),     as the arguments between (3.1) and (3.2), we take any $y\in \Om$ and choose  $z\in \pom$ such that $dist(y,z)=d_{y}.$
 Since the domain $\Om$ is  convex,  we may assume $z=0$, and take the line determined by $z$ and $y$
  as the $x_{n}-axis$ such that $\Om \subseteq R_{+}^n$. Without loss of generality, we assume $d_{\Om}\leq 1$. Otherwise, we can use the transform
$$\widetilde{x} =\frac{x}{d_{\Om} },  \ \  \widetilde{u}(\widetilde{x})=\frac{u(x)}{d_{\Om}}$$
to arrive at the assumed case.
  Let $$U(x)=(n+1)^2x_n^{\frac{1}{n+1}}-x_n^{2-\frac{1}{n+1}}.$$
 Then $$U(1+|\nabla U|^2)F[U]=UU_{nn}+n(1+U_n^2).$$
  By direct computation we have
 \begin{equation*}
 \begin{split}
  U_n&=(n+1)x_n^{\frac{1}{n+1}-1}-(2-\frac{1}{n+1})x_n^{1-\frac{1}{n+1}} \\
  U_{nn}&=-nx_n^{\frac{1}{n+1}-2}-(2-\frac{1}{n+1})(1-\frac{1}{n+1})x_n^{-\frac{1}{n+1}} .
\end{split}
 \end{equation*}
 Since $d_{\Om}\leq 1$ and $x_n\in [0, 1]$,  we have
 $$U> 0 \ \ on \ \ \pom$$ and
 \begin{equation*}
 \begin{split}
  U_n^2&\leq(n+1)^2 x_n^{\frac{-2n}{n+1}}- 4n+2\\
 U U_{nn}&\leq -n(n+1)^2 x_n^{\frac{-2n}{n+1}}- 2n^2+2.
\end{split}
 \end{equation*}
 Hence, $$U(1+|\nabla U|^2)F[U]\leq  -6n^2+3n+2<0$$
 which implies $U$ is a supper-solution to problem (1.1). Therefore
 $$0\leq u(y)\leq U(y)\leq  (n+1)^2 x_n^{\frac{1}{n+1}}= (n+1)^2(d_y)^{\frac{1}{n+1}},$$
 which, together with Lemma 2.4, implies (3.10).
\end{proof}
\vskip 0.5cm
{\bf Proof of (i) of Theorem 1.2:} \  When $a\in (2, n+1]$, Theorem 3.1 is   Theorem 1.2 exactly.  When $a\in (n+1, +\infty]$,
 Theorem 1.2   follows from  Theorem 3.2.

\section{ Proof of  Theorem 1.3}
%sec 4

In this section, we assume that $u$ is a concave solution to problem (1.1) and $x_{0}\in\pom$ is $(a, \eta)$- type point with $a\in [1,2)$ satisfying (1.5).  We are going to  prove Theorem 1.3.

Since $(1, \eta)$- type point is of course $(1+\varepsilon, \eta(\varepsilon))$- type point for any $\varepsilon>0$,
it is sufficient to prove (1.6) for any $a\in (1, 2)$.  Hence, from now on we assume $a\in (1, 2)$.

We will choose a  small positive number  $A$ such
that
\beq \label{4.1}A\leq  \eta^{\frac{1}{a-1}}d_{\Omega}.
\eeq
    Let
$$\widetilde{x}=A\frac{x}{d_{\Omega} } , \ \ \widetilde{u}(\widetilde{x})=A\frac{u(x)}{d_{\Omega} }.$$
We use $\widetilde{\Om}$ to denote the image of $\Om$
under this transform. It is easy to check $\widetilde{u}$ is the solution
to problem (1.1) in the domain  $\widetilde{\Om}$ which satisfies 
  \beq \label{4.2}  d_{\widetilde{\Omega}} \leq A.\eeq 
   Since $x_{0}=0\in\pom$, by  the convexity and (4.2) we see that
$\widetilde{\Om}\subseteq\Om$.  Note that the surface $x_{n}=\eta|x'|^{a}$ is transformed to the surface
$$\widetilde{x_{n}}=\eta(\frac{d_{\Omega}}{A})^{a-1}|\widetilde{x'}|^{a}.$$
 Then by (4.1)we have
\beq \label {4.3}\widetilde{\Om}\subseteq \{(x', x_{n} )\in R^n| \eta(\frac{d_{\Omega}}{A})^{a-1}|x'|^{a}
\leq x_{n}\leq A\}\subseteq\{(x', x_{n} )\in R^n| |x'|^{a} \leq x_{n}\leq A\}.
\eeq

 For brevity,  we denote  $\widetilde{u}$ and $\widetilde{\Om}$ still by $ u$ and $ \Om$, respectively, in the following.
Let
\beq \label{4.4} W(x)=((x_{n})^{\frac{2}{a}}-x_{1}^{2}-...-x_{n-1}^{2})^{\frac{1}{b}}
\eeq
where $b\in(2, 3)$ can be arbitrary constant.  We will prove $W$ is a super-solution to problem (1.1) in $\Om$.

Taking $\varepsilon=1$ in (3.3) we obtain the function as in (4.4). Hence have (3.7)-(3.8) where $\varepsilon=1$.
In this case, by (4.4) we have
 \begin{equation*}
 \begin{split}
J_{1}&=\frac{4(n+1-b))}{b^{2}}W^{1-2b}\cdot r^{2}\\
J_{2}&= \frac{2-2n}{b}W^{1-b}\\
J_{3}&=\frac{8(1-n)}{a^{2}b^{3}}W^{3-3b}\cdot x_{n}^{\frac{4}{a}-2}\\
J_{4}&=\frac{4(n+1-b))}{a^{2}b^{2}}W^{1-2b}\cdot x_{n}^{\frac{4}{a}-2}\\
J_{5}&=\frac{2(2-a)}{a^{2}b}W^{1-b}\cdot x_{n}^{\frac{2}{a}-2}\\
J_{6}&=\frac{8(2-a)}{a^{2}b^{3}}W^{3-3b}\cdot r^{2}\cdot x_{n}^{\frac{2}{a}-2}\\
J_{7}&=(n-2)(\frac{-8}{b^{3}})W^{3-3b}\cdot r^{2}\\
J_{8}&=nW^{-1}.
\end{split}
 \end{equation*}
Recalling that  $a\in (1,2)$, $b\in (2,3)$, $|x'|^{a} \leq x_{n}\leq A$ in $\Om$ by (4.3) and
$W^b(x)=x_{n}^{\frac{2}{a}}-|x'|^{2}$ by (4.4),
we obtain
\begin{equation*}
 \begin{split}
J_{1}&\leq\frac{4(n+1-b)}{b^{2}}W^{1-2b}\cdot x_{n}^{\frac{2}{a}}\\
&=\frac{4(n+1-b)}{b^{2}}W^{b-2}\cdot W^{3-3b}\cdot x_{n}^{\frac{4}{a}-2}\cdot x_{n}^{-\frac{2}{a}+2}\\
J_{2}&\leq 0\\
J_{3}&=\frac{8(1-n)}{a^{2}b^{3}}W^{3-3b}\cdot x_{n}^{\frac{4}{a}-2}\\
J_{4}&=\frac{4(n+1-b)}{a^{2}b^{2}}W^{b-2}\cdot W^{3-3b}\cdot x_{n}^{\frac{4}{a}-2}\\
J_{5}&=\frac{2(2-a)}{a^{2}b}W^{1-2b}\cdot W^{b}\cdot x_{n}^{\frac{2}{a}-2}\\
      &\leq\frac{2(2-a)}{a^{2}b}W^{1-2b}\cdot x_{n}^{\frac{4}{a}-2}\\
      &=\frac{2(2-a)}{a^{2}b}W^{b-2}\cdot W^{3-3b}\cdot x_{n}^{\frac{4}{a}-2}
\end{split}
 \end{equation*}

 \begin{equation*}
 \begin{split}
 J_{6}&\leq\frac{8(2-a)}{a^{2}b^{3}}W^{3-3b}\cdot x_{n}^{\frac{4}{a}-2}\\
J_{7} &\leq 0\\
J_{8}&= nW^{3-3b}\cdot W^{b(2-a)}\cdot W^{(a+1)b-4}\\
     &\leq nW^{3-3b}\cdot x_{n}^{\frac{4}{a}-2}\cdot W^{(a+1)b-4}.
\end{split}
 \end{equation*}

Hence, in $\Om$ we have
\begin{equation*}
 \begin{split}
J_{3}+J_{6}&\leq\frac{-8(n+a-3)}{a^{2}b^{3}}W^{3-3b}\cdot x_{n}^{\frac{4}{a}-2}\\
I+J&=J_{1}+J_{2}+J_{3}+J_{4}+J_{5}+J_{6}+J_{7}+J_{8}\\
&\leq J_{1}+J_{3}+J_{4}+J_{5}+J_{6}+J_{8}\\
&\leq [ \frac{4(n+1-b)}{b^{2}}W^{b-2} x_{n}^{-\frac{2}{a}+2}-\frac{8(n+a-3)}{a^{2}b^{3}}+
\frac{4(n+1-b)}{a^{2}b^{2}}W^{b-2}\\
& \ +\frac{2(2-a)}{a^{2}b}W^{b-2}
  +nW^{(a+1)b-4}] \cdot W^{3-3b}\cdot x_{n}^{\frac{4}{a}-2}\\
&\leq[ \frac{4(n+1-b)}{b^{2}}W^{b-2} A^{-\frac{2}{a}+2}+\frac{-8(n+a-3)}{a^{2}b^{3}}+
\frac{4n+4-2ab}{a^{2}b^{2}}W^{b-2}\\
&\ +nW^{b-2+ab-2}]
 \cdot W^{3-3b}\cdot x_{n}^{\frac{4}{a}-2}\\
 &=\{W^{b-2}\cdot[\frac{4(n+1-b)}{b^{2}} A^{-\frac{2}{a}+2}+
\frac{4n+4-2ab}{a^{2}b^{2}}+nW^{ab-2}]\\
& \ +\frac{-8(n+a-3)}{a^{2}b^{3}}\}
\cdot W^{3-3b}\cdot x_{n}^{\frac{4}{a}-2}\\
&:=\Phi \cdot W^{3-3b}\cdot x_{n}^{\frac{4}{a}-2}.
\end{split}
 \end{equation*}

Since
  $$0\leq W(x)\leq x_{n}^{\frac{2}{ab}}\leq A^{\frac{2}{ab}}$$
  by (4.3) and (4.4),  we obtain
\begin{equation*}
 \begin{split}
 \Phi & \leq A^{\frac{2(b-2)}{ab}}\cdot[\frac{4(n+1-b)}{b^{2}}A^{-\frac{2}{a}+2}+
\frac{4n+4-2ab}{a^{2}b^{2}}+nA^{\frac{2(ab-2)}{ab}}]+\frac{-8(n+a-3)}{a^{2}b^{3}}\\
 &= [\frac{4(n+1-b)}{b^{2}}A^{2-\frac{4}{ab}}+
\frac{4n+4-2ab}{a^{2}b^{2}}A^{\frac{2(b-2)}{ab}}+nA^{\frac{2(ab+b-4)}{ab}}]+\frac{-8(n+a-3)}{a^{2}b^{3}}.
\end{split}
 \end{equation*}
Furthermore, it follows from the fact $n\geq 2, a\in (1,2)$ and $b\in (2,3)$ that
 $$-8(n+a-3)<0,\ \ 2-\frac{4}{ab}>0,\ \ \frac{2(b-2)}{ab}>0,\ \ \frac{2(ab+b-4)}{ab}>0.$$
Therefore,  we can take $A>0$ is small enough in advance such that $\Phi \leq 0$. Hence
$$F[W]=\frac{I+J}{1+W_r^2+W_n^2}\leq 0.$$
In this way, we have proved that $W$ is a super-solution to problem (1.1) in  $\Om.$
By comparison principal, $0\leq u\leq W$.  Restricting this  on to $x_{n}-axis$, we obtain that
$$0\leq u(\mathbf{0},x_{n})\leq W(\mathbf{0},x_{n})=
x_{n}^{\frac{2}{ab}} $$
  for any $b\in (2,3)$, which implies the desired (1.6) for any $\delta\in (0, \frac{1}{2})$.
Moreover, this implies (1.6) for any $\delta\geq \frac{1}{2}$. The proof of  Theorem 1.3 is finished.

%\vskip 1cm
%\noindent{\bf Remark 4.1}When $a=1$, n=2, u is Lipschitz along $r$ direction.This can be proved as following, by definition, through some translations and rotations, we assume\\
%(1)$\Om\subseteq R_{+}^{n}$ with $x_{0}=0$.\\
%(2)$\varphi_{x_{0}}(x')\geq\eta|x'|$, when $|x'|\leq\mu$.\\

%By the convex of $\Omega$, in fact, we have $\Omega\subseteq\{(x', x_{n}): \varphi_{x_{0}}(x')\geq\eta|x'|\}$,\\

%$\{(x', x_{n}): \varphi_{x_{0}}(x')\geq\eta|x'|\}=\{(x', x_{n}): \varphi_{x_{0}}(x')\geq\eta\sqrt{x_{1}^{2}+x_{2}^{2}}\}:=V$\\

%Where $V$ is a two-dimension cone. By the result of [HSW1], the solution in $V$ have the form $w=rh(\theta)$ where $r=\sqrt{x_{1}^{2}+x_{2}^{2}}$,
%$\theta=\arctan\frac{x_{2}}{x_{1}}$. Then  $w$ is a subsolution on $\Omega$.  $u\leq w=rh(\theta)$. And we can see that $u$ is one-dimension convex function along $r$ direction. By lemma 2.4, we concluded u is Lipschitz along $r$ direction.


\newpage

\begin{thebibliography}{999}

\parskip2.5pt

\bibitem {[A1]}Anderson, A., Complete minimal varieties in hyperbolic space, Invent. Math., 69 (1982), 477-494.

\bibitem {[A2]}Anderson, A., Complete minimal hypersurfaces in hyperbolic $n$-manifolds, Comment. Math. Helv.,
           58 (1983),264-290.


%\bibitem {[GT]} Gilbarg, D.,   Trudinger, N.S.,
              %Elliptic partial differential equations of second order,
              %Springer-Verlag, New York, 1983.
\bibitem {[GS]}Guan B., Spruck, J., Hypersurfaces of constant mean curvature in the hyperbolic space
with prescribed asymptotic boundary at infinity, Amer. J. Math. 122 (2000), 1036-1060.

  %\bibitem  {[HJ]} Han, Q., Jiang, X.,Boundary expansions for   minimal graphs in the hyperbolic space,
                %arXiv:1412.7608.

\bibitem  {[HSW]}  Han,Q.,  Shen,W.,   Wang, Y.,
                Optimal regularity of minimal graphs in the hyperbolic space,
                 Calc. Var. Partial Differential Equations, 55(2016), 1-19.

%\bibitem  {[HSW1]}  Qing Han, Weiming Shun, Yue Wang,
                %minimal graphs in the hyperbolic space with singular asymptotic bounders,
                 %preprint.
\bibitem   {[HL]} Hardt, R., Lin, F.H, Reguloarity at infinity for area-minimizing hypersurfaces in hyperbolic space,
            Invent. Math., 88 (1987), 217-224.
  %\bibitem {[JW]} Jian, H.Y.,  Wang, X.-J.,
              %Bernstein theorem and regularity for a class of  Monge-Amp\`ere equation,
              %J. Diff. Geom. 93 (2013),431-469.

\bibitem {[JL]} Jian, H.Y.,  Li, Y.,
              Optimal boundary regularity for a singular  Monge-Amp\`ere equation,  J. Differential Equations, 264 (2018), 6873-6890.

\bibitem  {[L1]}   Lin, F.H., On the Dirichlet problem for   minimal graphs in the hyperbolic space,
             Invent. Math., 96 (1989), Invent. Math., 96 (1989), 593-612.

\bibitem  {[L2]}   Lin, F.H., Asymtotic behavior of  area-minimizing currents in hyperbolic space,
            Comm. pure Appl. Math., 42 (1989),229-242
            
\bibitem  {[S]} Serre, D., Multidimensional shock interaction for a Chaplygin gas, Arch. Rational Mech. Anal., 
191 (2009), 539-577. 

\bibitem  {[T]} Tonegawa, Y., Existence and regularity of constant mean curvature hypersurfaces in hyperbolic space,
Math. Z., 221 (1996), 591-615.

\end{thebibliography}
\end{document}